\documentclass{amsart}
\usepackage{amsfonts}
\usepackage{amssymb}
\usepackage{amsmath}

\newtheorem{theorem}{Theorem}[section]
\newtheorem*{theorem A}{Theorem A}
\newtheorem*{theorem B}{N\"olker's Theorem}

\theoremstyle{remark}

\theoremstyle{remark}

\theoremstyle{definition}

\numberwithin{equation}{section}
\def \({\left ( }
\def \){\right )}
\def \<{\left < }
\def \>{\right >}

\input{tcilatex}

\begin{document}
\title{\textbf{Semi-parallel Meridian Surfaces in }$\mathbb{E}^{4}$}
\author{Bet\"{u}l Bulca \& Kadri Arslan}
\address{Uluda\u{g} University, Art and Science Faculty, Department of
Mathematics, Bursa-TURKEY}
\email{bbulca@uludag.edu.tr; arslan@uludag.edu.tr}
\thanks{}
\date{March 25, 2015 }
\dedicatory{}

\begin{abstract}
In the present article we study a special class of surfaces in the
four-dimensional Euclidean space, which are one-parameter systems of
meridians of the standard rotational hypersurface. They are called meridian
surfaces. We classified semi-parallel meridian surface in $4$-dimensional
Euclidean space $\mathbb{E}^{4}.$
\end{abstract}

\maketitle

\section{Introduction}

Let $M$ be a submanifold of a $n$-dimensional Euclidean space $\mathbb{E}%
^{n}.$ Denote by $\overline{R}$ the curvature tensor of the Vander
Waerden-Bortoletti connection $\overline{\nabla }$ of $M$ and $h$ is the
second fundamental form of $M$ in $\mathbb{E}^{n}.$ The submanifold $M$ is
called semi-parallel (or semi-symmetric \cite{Lu}) if $\overline{R}\cdot h=0$
\cite{De}. This notion is an extrinsic analogue for semi-symmetric spaces,
i.e. Riemannian manifolds for which $R\cdot R=0$ and a direct generalization
of parallel submanifolds, i.e. submanifolds for which $\overline{\nabla }%
h=0. $ In \cite{De} J. Deprez showed the fact that the submanifold $M\subset 
\mathbb{E}^{n}$ is semi-parallel implies that $(M,g)$ is semi-symmetric. For
references on semi-symmetric spaces, see \cite{Sz}; for references on
parallel immersions, see \cite{Fe}. In \cite{De} J. Deprez gave a local
classification of semi-parallel hypersurfaces in Euclidean $n$-space $%
\mathbb{E}^{n}$.

Recently, the present authors considered the Wintgen ideal surfaces in
Euclidean $n$-space $\mathbb{E}^{n}.$ They showed that Wintgen ideal
surfaces in $\mathbb{E}^{n}$ satisfying the semi-parallelity condition 
\begin{equation}
\overline{R}(X,Y)\cdot h=0  \label{a.1}
\end{equation}%
are of flat normal connection \cite{BA1}. Further, \ the same authors in 
\cite{BA2} proved that the tensor product surfaces in $\mathbb{E}^{4}$
satisfying the semi-parallelity condition (\ref{a.1}) are totally umbilical.

In \cite{GM5} Ganchev and Milousheva constructed special two dimensional
surfaces which are one-parameter of meridians of the rotation hypersurfaces
in $\mathbb{E}^{4}$ and called these surfaces \textit{meridian surfaces. }%
The geometric construction of the meridian surfaces is different from the
construction of the standard rotational surfaces with two dimensional axis
in $\mathbb{E}^{4}$ \cite{GM1}$.$ The same authors classified the meridian
surfaces with constant Gauss curvature ($K\neq 0$) and constant mean
curvature $H$ \cite{GM5}. Recently, meridian surfaces with $1$-type Gauss
map is characterized by the present authors and Milousheva in \cite{BAM}.
Further, meridian surfaces were studied in \cite{GM2} as a surface in
Minkowski $4$-space. For more details see also \cite{GM3}, \cite{GM4} and 
\cite{OBBA}.

In the present study we consider the meridian surfaces in $4$-dimensional
Euclidean space $\mathbb{E}^{4}.$ We gave a classification of the meridian
surfaces in $4$-dimensional Euclidean space $\mathbb{E}^{4}$ satisfying the
semi-parallelity condition (\ref{a.1}).

\section{Basic Concepts}

Let $M$ be a smooth surface in n-dimensional Euclidean space $\mathbb{E}^{n}$
given with the surface patch $X(u,v)$ : $(u,v)\in D\subset \mathbb{E}^{2}$.
The tangent space to $M$ at an arbitrary point $p=X(u,v)$ of $M$ span $%
\left
\{ X_{u},X_{v}\right \} $. In the chart $(u,v)$ the coefficients of
the first fundamental form of $M$ are given by 
\begin{equation}
E=\left \langle X_{u},X_{u}\right \rangle ,F=\left \langle X_{u},X_{v}\right
\rangle ,G=\left \langle X_{v},X_{v}\right \rangle ,  \label{B1}
\end{equation}%
where $\left \langle ,\right \rangle $ is the Euclidean inner product. We
assume that $W^{2}=EG-F^{2}\neq 0,$ i.e. the surface patch $X(u,v)$ is
regular.\ For each $p\in M$, consider the decomposition $T_{p}\mathbb{E}%
^{n}=T_{p}M\oplus T_{p}^{\perp }M$ where $T_{p}^{\perp }M$ is the orthogonal
component of the tangent plane $T_{p}M$ in $\mathbb{E}^{n}$, that is the
normal space of $M$ at $p$.

Let $\chi (M)$ and $\chi ^{\perp }(M)$ be the space of the smooth vector
fields tangent and normal to $M$ respectively. Denote by $\nabla $ and $%
\overline{\nabla }$ the Levi-Civita connections on $M$ and $\mathbb{E}^{n},$
respectively. Given any vector fields $X_{i}$ and $X_{j}$ tangent to $M$
consider the second fundamental map $h:\chi (M)\times \chi (M)\rightarrow
\chi ^{\perp }(M);$%
\begin{equation}
h(X_{i},X_{j})=\widetilde{\nabla }_{X_{i}}X_{j}-\nabla _{X_{i}}X_{j};\text{ }%
1\leq i,j\leq 2.  \label{B2}
\end{equation}%
where $\widetilde{\nabla }$ is the induced. This map is well-defined,
symmetric and bilinear.

For any normal vector field $N_{\alpha }$ $1\leq \alpha \leq n-2$ of $M$,
recall the shape operator $A:\chi ^{\perp }(M)\times \chi (M)\rightarrow
\chi (M);$%
\begin{equation}
A_{N_{\alpha }}X_{i}=-\widetilde{\nabla }_{N_{\alpha
}}X_{i}+D_{X_{i}}N_{\alpha };\text{ \  \ }1\leq i\leq 2.  \label{B3}
\end{equation}%
where $D$ denotes the normal connection of $M$ in $\mathbb{E}^{n}$ \cite{Ch1}%
$.$This operator is bilinear, self-adjoint and satisfies the following
equation:%
\begin{equation}
\left \langle A_{N_{\alpha }}X_{i},X_{j}\right \rangle =\left \langle
h(X_{i},X_{j}),N_{\alpha }\right \rangle \text{, }1\leq i,j\leq 2.
\label{B4}
\end{equation}

The equation (\ref{B2}) is called Gaussian formula, and%
\begin{equation}
h(X_{i},X_{j})=\overset{n-2}{\underset{\alpha =1}{\sum }}h_{ij}^{\alpha
}N_{\alpha },\  \  \  \  \ 1\leq i,j\leq 2  \label{B5}
\end{equation}%
where $h_{ij}^{\alpha }$ are the coefficients of the second fundamental form 
$h$ \cite{Ch1}. If $h=0$ then $M$ is called totally geodesic. $M$ is totally
umbilical if all shape operators are proportional to the identity map. $M$
is an isotropic surface if for each $p$ in $M$, $\left \Vert
h(X,X)\right
\Vert $ is independent of the choice of a unit vector $X$ in $%
T_{p}M$.

If we define a covariant differentiation $\overline{\nabla }h$ of the second
fundamental form $h$ on the direct sum of the tangent bundle and normal
bundle $TM\oplus T^{\bot }M$ of $M$ by 
\begin{equation}
(\overline{\nabla }_{X_{i}}h)(X_{j},X_{k})=D_{X_{i}}h(X_{j},X_{k})-h(\nabla
_{X_{i}}X_{j},X_{k})-h(X_{j},\nabla _{X_{i}}X_{k}),  \label{B6}
\end{equation}%
for any vector fields $X_{i}$,$X_{j},X_{k}$ tangent to $M$. Then we have the
Codazzi equation 
\begin{equation}
(\overline{\nabla }_{X_{i}}h)(X_{j},X_{k})=(\overline{\nabla }%
_{X_{j}}h)(X_{i},X_{k}),  \label{B7}
\end{equation}%
where $\overline{\nabla }$ is called the Vander Waerden-Bortoletti
connection of $M$ \cite{Ch1}.

We denote $R$ and $\overline{R}$ the curvature tensors associated with $%
\nabla $ and $D$ respectively;%
\begin{eqnarray}
R(X_{i},X_{j})X_{k} &=&\nabla _{X_{i}}\nabla _{X_{j}}X_{k}-\nabla
_{X_{j}}\nabla _{X_{i}}X_{k}-\nabla _{\lbrack X_{i},X_{j}]}X_{k},  \label{B8}
\\
R^{\bot }(X_{i},X_{j})N_{\alpha } &=&h(X_{i},A_{N_{\alpha
}}X_{j})-h(X_{j},A_{N_{\alpha }}X_{i}).  \label{B9}
\end{eqnarray}

The equation of Gauss and Ricci are given respectively by%
\begin{equation}
\text{ \  \ }\left \langle R(X_{i},X_{j})X_{k},X_{l}\right \rangle =\left
\langle h(X_{i},X_{l}),h(X_{j},X_{k})\right \rangle -\left \langle
h(X_{i},X_{k}),h(X_{j},X_{l})\right \rangle ,  \label{B10}
\end{equation}%
\begin{equation}
\text{\ }\left \langle R^{\bot }(X_{i},X_{j})N_{\alpha },N_{\beta }\right
\rangle =\left \langle [A_{N_{\alpha }},A_{N_{\beta }}]X_{i,}X_{j}\right
\rangle ,  \label{B11}
\end{equation}%
for the vector fields $X_{i},X_{j},X_{k}$ tangent to $M$ and $N_{\alpha
},N_{\beta }$ normal to $M$ \cite{Ch1}.

Let us $X_{i}\wedge X_{j}$ denote the endomorphism $X_{k}\longrightarrow
\left \langle X_{j},X_{k}\right \rangle X_{i}-$ $\left \langle
X_{i},X_{k}\right \rangle X_{j}.$ Then the curvature tensor $R$ of $M$ is
given by the equation 
\begin{equation}
R(X_{i},X_{j})X_{k}=\overset{n-2}{\underset{\alpha =1}{\sum }}\left(
A_{N_{\alpha }}X_{i}\wedge A_{N_{\alpha }}X_{j}\right) X_{k}.  \label{B12}
\end{equation}

It is easy to show that 
\begin{equation}
R(X_{i},X_{j})X_{k}=K\left( X_{i}\wedge X_{j}\right) X_{k}.  \label{B13}
\end{equation}%
where $K$ is the Gaussian curvature of $M$ defined by 
\begin{equation}
K=\left \langle h(X_{1},X_{1}),h(X_{2},X_{2})\right \rangle -\left \Vert
h(X_{1},X_{2})\right \Vert ^{2}  \label{B14}
\end{equation}%
(see \cite{GR}).

The normal curvature $K_{N}$ of $M$ is defined by (see \cite{DDVV}) 
\begin{equation}
K_{N}=\left \{ \overset{n-2}{\underset{1=\alpha <\beta }{\sum }}\left
\langle R^{\bot }(X_{1},X_{2})N_{\alpha },N_{\beta }\right \rangle
^{2}\right \} ^{1/2}.  \label{B15}
\end{equation}

We observe that the normal connection $D$ of $M$ is flat if and only if $%
K_{N}=0,$ and by a result of Cartan, this equivalent to the
diagonalisability of all shape operators $A_{N_{\alpha }}$\ of $M$, which
means that $M$ is a totally umbilical surface in $\mathbb{E}^{n}$.

\section{Semi-parallel Surfaces}

Let $M$ a smooth surface in $n$-dimensional Euclidean space $\mathbb{E}^{n}.$
Let $\overline{\nabla }$ be the connection of Vander Waerden-Bortoletti of $%
M $. Denote the tensors $\overline{\nabla }$ by $\overline{R}$ . Then the
product tensor $\overline{R}\cdot h$ of the curvature tensor $\overline{R}$
with the second fundamental form $h$ is defined by%
\begin{eqnarray*}
(\overline{R}(X_{i},X_{j})\cdot h)(X_{k},X_{l}) &=&\overline{\nabla }%
_{X_{i}}(\overline{\nabla }_{X_{j}}h(X_{k},X_{l}))-\overline{\nabla }%
_{X_{j}}(\overline{\nabla }_{X_{i}}h(X_{k},X_{l})) \\
&&-\overline{\nabla }_{[X_{i},X_{j}]}h(X_{k},X_{l}),
\end{eqnarray*}%
for all $X_{i},X_{j},X_{k},X_{l}$ tangent to $M.$

The surface $M$ is said to be semi-parallel if $\overline{R}\cdot h=0,$ i.e. 
$\overline{R}(X_{i},X_{j})\cdot h=0$ (\cite{Lu}, \cite{De}, \cite{Ds}, \cite%
{OAM}). It is easy to see that 
\begin{eqnarray}
(\overline{R}(X_{i},X_{j})\cdot h)(X_{k},X_{l}) &=&R^{\bot
}(X_{i},X_{j})h(X_{k},X_{l})  \label{d1} \\
&&\text{-}h(R(X_{i},X_{j})X_{k},X_{l})\text{-}h(X_{k},R(X_{i},X_{j})X_{l}), 
\notag
\end{eqnarray}

This notion is an extrinsic analogue for semi-symmetric spaces, i.e.
Riemannian manifolds for which $R\cdot R=0$ and a generalization of parallel
surfaces, i.e. $\overline{\nabla }h=0$ \cite{Fe}$.$

Substituting (\ref{B5}) and (\ref{B4}) into (\ref{B9}) we get 
\begin{equation}
R^{\bot }(X_{1},X_{2})N_{\alpha }=h_{12}^{\alpha
}(h(X_{1},X_{1})-h(X_{2},X_{2})+(h_{22}^{\alpha }-h_{11}^{\alpha
})h(X_{1},X_{2}).  \label{d2}
\end{equation}%
Further, by the use of (\ref{B13}) we get 
\begin{equation}
R(X_{1},X_{2})X_{1}=-KX_{2},R(X_{1},X_{2})X_{2}=KX_{1}.  \label{d3}
\end{equation}%
So, substituting (\ref{d2}) and (\ref{d3}) into (\ref{d1}) we obtain%
\begin{eqnarray}
(\overline{R}(X_{1},X_{2})\cdot h)(X_{1},X_{1}) &=&\left( \sum_{\alpha
=1}^{n-2}h_{11}^{\alpha }(h_{22}^{\alpha }-h_{11}^{\alpha })+2K\right)
h(X_{1},X_{2})  \notag \\
&&+\sum_{\alpha =1}^{n-2}h_{11}^{\alpha }h_{12}^{\alpha
}(h(X_{1},X_{1})-h(X_{2},X_{2})),  \notag \\
(\overline{R}(X_{1},X_{2})\cdot h)(X_{1},X_{2}) &=&\left( \sum_{\alpha
=1}^{n-2}h_{12}^{\alpha }(h_{22}^{\alpha }-h_{11}^{\alpha })\right)
h(X_{1},X_{2})  \label{D2} \\
&&+(\sum_{\alpha =1}^{n-2}h_{12}^{\alpha }h_{12}^{\alpha }\text{-}%
K)(h(X_{1},X_{1})\text{-}h(X_{2},X_{2})),  \notag \\
(\overline{R}(X_{1},X_{2})\cdot h)(X_{2},X_{2}) &=&\left( \sum_{\alpha
=1}^{n-2}h_{22}^{\alpha }(h_{22}^{\alpha }-h_{11}^{\alpha })-2K\right)
h(X_{1},X_{2})  \notag \\
&&+\sum_{\alpha =1}^{n-2}h_{22}^{\alpha }h_{12}^{\alpha
}(h(X_{1},X_{1})-h(X_{2},X_{2})).  \notag
\end{eqnarray}%
\qquad Semi-parallel surfaces in $\mathbb{E}^{n}$ are classified by J.
Deprez \cite{De}:

\begin{theorem}
\cite{De} Let $M$ a surface in $n$-dimensional Euclidean space $\mathbb{E}%
^{n}.$ Then $M$ is semi-parallel if and only if locally;

i) $M$ is equivalent to a $2$-sphere, or

ii) $M$ has trivial normal connection, or

iii) $M$ is an isotropic surface in $\mathbb{E}^{5}\subset \mathbb{E}^{n}$
satisfying $\left \Vert H\right \Vert ^{2}=3K.$
\end{theorem}

\section{Meridian Surfaces in $\mathbb{E}^{4}$}

In the following sections, we will consider the meridian surfaces in $%
\mathbb{E}^{4}$ which is first defined by Ganchev and Milousheva \cite{GM1}.
The meridian surfaces are one-parameter systems of meridians of the standard
rotational hypersurface in $\mathbb{E}^{4}$.

Let $\{e_{1},e_{2},e_{3},e_{4}\}$ be the standard orthonormal frame in $%
\mathbb{E}^{4}$, and $S^{2}(1)$ be a 2-dimensional sphere in $\mathbb{E}%
^{3}=span\{e_{1},e_{2},e_{3}\}$, centered at the origin $O$. We consider a
smooth curve $C:r=r(v),\,v\in J,\, \,J\subset \mathbb{R}$ on $S^{2}(1)$,
parameterized by the arc-length ($\left \Vert (r^{\prime
}{})^{2}(v)\right
\Vert =1$). We denote $t=r^{\prime }$ and consider the
moving frame field $\{t(v),n(v),r(v)\}$ of the curve $C$ on $S^{2}(1)$. With
respect to this orthonormal frame field the following Frenet formulas hold
good: 
\begin{equation}
\begin{array}{l}
\vspace{2mm}r^{\prime }=t; \\ 
\vspace{2mm}t^{\prime }=\kappa \,n-r; \\ 
\vspace{2mm}n^{\prime }=-\kappa \,t,%
\end{array}
\label{c1}
\end{equation}%
where $\kappa $ is the spherical curvature of $C$.

Let $f=f(u),\, \,g=g(u)$ be smooth functions, defined in an interval $%
I\subset \mathbb{R}$, such that

\begin{equation}
(f^{\prime })^{2}(u)+(g^{\prime })^{2}(u)=1,\, \,u\in I.  \label{c2}
\end{equation}

In \cite{GM1} Ganchev and Milousheva constructed a surface $M^{2}$ in $%
\mathbb{E}^{4}$ in the following way: 
\begin{equation}
M^{2}:X(u,v)=f(u)\,r(v)+g(u)\,e_{4},\quad u\in I,\,v\in J.  \label{c3}
\end{equation}

The surface $M^{2}$ lies on the rotational hypersurface $M^{3}$ in $\mathbb{E%
}^{4}$ obtained by the rotation of the meridian curve $\alpha :u\rightarrow
(f(u),g(u))$ around the $Oe_{4}$-axis in $\mathbb{E}^{4}$. Since $M^{2}$
consists of meridians of $M^{3}$, we call $M^{2}$ a \textit{meridian surface 
}\cite{GM1}. If we denote by $\kappa _{\alpha }$ the curvature of meridian
curve $\alpha ,$ i.e., 
\begin{equation}
\kappa _{\alpha }=f^{\prime }(u)g^{\prime \prime }(u)-f^{\prime \prime
}(u)g(u)=\frac{-f^{\prime \prime }(u)}{\sqrt{1-f^{\prime 2}(u)}}.  \label{C4}
\end{equation}

We consider the following orthonormal moving frame fields, $%
X_{1},X_{2},N_{1},N_{2}$ on the meridian surface $M^{2}$ such that $%
X_{1},X_{2}$ are tangent to $M^{2}$ and $N_{1},N_{2}$ are normal to $M^{2}$.
The tangent space of $M^{2}$ is spanned by the vector fields: 
\begin{equation}
\begin{array}{l}
\vspace{2mm}\vspace{2mm}X_{1}=\frac{\partial X}{\partial u},\text{ \  \ }%
X_{2}=\frac{1}{f}\frac{\partial X}{\partial v}, \\ 
N_{1}=n(v),\text{ \ }N_{2}=-g^{\prime }(u)\,r(v)+f^{\prime }(u)\,e_{4}.%
\end{array}
\label{c4}
\end{equation}

By a direct computation we have the components of the second fundamental
forms as;%
\begin{equation}
\begin{array}{l}
\vspace{2mm}%
\begin{array}{ll}
h_{11}^{1}=h_{12}^{1}=h_{21}^{1}=0, & h_{22}^{1}={\frac{\kappa }{f}},%
\end{array}
\\ 
\begin{array}{lll}
h_{11}^{2}=\kappa _{\alpha } & h_{12}^{2}=h_{21}^{2}=0, & h_{22}^{2}={\frac{%
g^{\prime }}{f}}.%
\end{array}%
\end{array}
\label{c5}
\end{equation}

Therefore the shape operator matrices of $M^{2}$ are of the form%
\begin{equation}
A_{N_{_{1}}}=\left[ 
\begin{array}{ll}
0 & 0 \\ 
0 & {\frac{\kappa }{f}}%
\end{array}%
\right] ,\text{ }A_{N_{_{2}}}=\left[ 
\begin{array}{ll}
\kappa _{\alpha } & 0 \\ 
0 & {\frac{g^{\prime }}{f}}%
\end{array}%
\right]  \label{c6}
\end{equation}%
and hence we have%
\begin{equation}
\begin{array}{l}
K={\frac{\kappa _{\alpha }g^{\prime }}{f}}, \\ 
\vspace{2mm}K_{N}=0,%
\end{array}
\label{c7}
\end{equation}%
which implies that the meridian surface $M^{2}$ is totally umbilical surface
in $\mathbb{E}^{4}.$

In \cite{GM5} Ganchev and Milousheva constructed three main classes of
meridian surfaces:

I. $\kappa =0;$ i.e. the curve $C$ is a great circle on $S^{2}(1)$. In this
case $N_{1}=const.$ and $M^{2}$ is a planar surface lying in the constant $3$%
-dimensional space spanned by $\left \{ X_{1},X_{2},N_{2}\right \} $.
Particularly, if in addition $\kappa _{\alpha }=0,$ i.e. the meridian curve
is a part of a straight line, then $M^{2}$ is a developable surface in the $%
3 $-dimensional space spanned by $\left \{ X_{1},X_{2},N_{2}\right \} $.

II. $\kappa _{\alpha }=0,$ i.e. the meridian curve is a part of a straight
line. In such case $M^{2}$ is a developable ruled surface. If in addition $%
\kappa =const.$, i.e. $C$ is a circle on $S^{2}(1)$, then $M^{2}$ is a
developable ruled surface in a $3$-dimensional space. If $\kappa \neq
const., $i.e. $C$ is not a circle on $S^{2}(1)$, then $M^{2}$ is a
developable ruled surface in $\mathbb{E}^{4}$.

III. $\kappa _{\alpha }\kappa \neq 0,$ i.e. $C$ is not a circle on $S^{2}(1)$
and $\alpha $ is not a straight line. In this general case the parametric
lines of $M^{2}$ given by (\ref{c3}) are orthogonal and asymptotic.

We proved the following Theorem.

\begin{theorem}
Let $M^{2}$ be a meridian surface in $\mathbb{E}^{4}$ given with the
parametrization (\ref{c3})$.$ Then $M^{2}$ is semi-parallel if and only if \
one of the following holds:

i) $M^{2}$ is a developable ruled surface in $\mathbb{E}^{3}$ or $\mathbb{E}%
^{4}$ which considered in Case II of the classification above,

ii) the curve $C$ is a circle on $S^{2}(1)$ with non-zero constant spherical
curvature and the meridian curve is determined by 
\begin{equation*}
f(u)=\pm \sqrt{u^{2}-2au+2b};\text{ }g(u)=-\sqrt{2b-a^{2}}\ln \left( u-a-%
\sqrt{u^{2}-2au+2b}\right) ,
\end{equation*}%
where $a=const,b=const.$ In this case $M^{2}$ is a planar surface lying in
3-dimensional space spanned by $\left \{ X_{1},X_{2},N_{2}\right \} $.
\end{theorem}

\begin{proof}
Let $M^{2}$ be a meridian surface in $\mathbb{E}^{4}$ given with the
parametrization (\ref{c3}). Then by the use of (\ref{B5}) with (\ref{c5}) we
see that%
\begin{eqnarray}
h(X_{1},X_{2}) &=&0,  \label{c8} \\
h(X_{1},X_{1})-h(X_{2},X_{2}) &=&-{\frac{\kappa }{f}N_{1}}+\left( \kappa
_{\alpha }-{\frac{g^{\prime }}{f}}\right) N_{2}.  \notag
\end{eqnarray}%
Further, substituting (\ref{c8}) and (\ref{c5}) into (\ref{D2}) and after
some computation one can get%
\begin{eqnarray*}
(\overline{R}(X_{1},X_{2})\cdot h)(X_{1},X_{1}) &=&0, \\
(\overline{R}(X_{1},X_{2})\cdot h)(X_{1},X_{2}) &=&-K\left( -{\frac{\kappa }{%
f}N_{1}}+\left( \kappa _{\alpha }-{\frac{g^{\prime }}{f}}\right)
N_{2}\right) , \\
(\overline{R}(X_{1},X_{2})\cdot h)(X_{2},X_{2}) &=&0.
\end{eqnarray*}%
Suppose that, $M^{2}$ is semi-parallel then by definition 
\begin{equation*}
(\overline{R}(X_{1},X_{2})\cdot h)(X_{i},X_{j})=0,\text{ }1\leq i,j\leq 2,
\end{equation*}%
is satisfied. So, we get%
\begin{equation*}
K\left( -{\frac{\kappa }{f}}N_{1}+\left( \kappa _{\alpha }-{\frac{g^{\prime }%
}{f}}\right) N_{2}\right) =0.
\end{equation*}%
Hence, two possible cases occur; $K=0$ or $\kappa =0$ and $\kappa _{\alpha }-%
{\frac{g^{\prime }}{f}}=0.$ For the first case $\kappa _{\alpha }=0,$ i.e.
the meridian curve is a part of a straight line. In such case $M^{2}$ is a
developable ruled surface given in the Case II. For the second case $\kappa
=0$ means that the curve $c$ is a great circle on $S^{2}(1).$ In this case $%
M^{2}$ lies in the 3-dimensional space spanned by $\left \{
X_{1},X_{2},N_{2}\right \} .$ Further, using (\ref{C4}) the equation $\kappa
_{\alpha }-{\frac{g^{\prime }}{f}}=0$ can be rewritten in the form%
\begin{equation*}
f(u)f^{\prime \prime }(u)-(f^{\prime }(u))^{2}+1=0,
\end{equation*}%
which has the solution 
\begin{equation}
f(u)=\pm \sqrt{u^{2}-2au+2b}.  \label{c9}
\end{equation}%
Consequently, by substituting (\ref{c9}) into (\ref{c2}) one can get 
\begin{equation*}
g(u)=-\sqrt{2b-a^{2}}\ln \left( u-a-\sqrt{u^{2}-2au+2b}\right) .
\end{equation*}%
This completes the proof of the theorem.
\end{proof}

\end{document}